\documentclass[a4paper,10pt]{amsart}
\usepackage[active]{srcltx}
\usepackage{mathrsfs}
\usepackage[all]{xy}
\usepackage{calrsfs}
\usepackage{times}
\usepackage[scaled=0.92]{helvet}
\usepackage{soul}

\usepackage{amsfonts}
\usepackage{amssymb}
\usepackage[usenames]{color}

\usepackage[dvistyle]{todonotes}

\usepackage[usenames]{color}
\newtheorem{theorem}{Theorem}

\newtheorem{lemma}[theorem]{Lemma}
\newtheorem{proposition}[theorem]{Proposition}
\newtheorem{corollary}[theorem]{Corollary}

\theoremstyle{definition}
\newtheorem{definition}[theorem]{Definition}
\newtheorem{remark}[theorem]{Remark}
\newtheorem{example}[theorem]{Example}

\definecolor{MyDarkBlue}{rgb}{0,0.08,0.60}

\newcommand{\A}{{ \rm Aut }}

\renewcommand{\O}{{\mathcal{O}}}
\newcommand{\Z}{{\mathbb{Z}}}

\newcommand{\lf}{\left\lfloor}
\newcommand{\rf}{\right\rfloor}

\renewcommand{\mod}{{\;\rm mod}}

\usepackage{amsmath,amssymb,amsfonts,amsthm}

\title[Automorphisms of Curves and Weierstrass semigroups for HKG-covers]{Automorphisms of Curves and Weierstrass semigroups for Harbater-Katz-Gabber covers}

\date{\today}

\author{Sotiris Karanikolopoulos }
\address{Freie Universit\"at Berlin \\ Institut f\"ur Mathematik,
Arnimallee 3,
14195 Berlin, Germany}
\email{skaran@zedat.fu-berlin.de}

\author{Aristides Kontogeorgis}
\address{ National and Kapodistrian University of Athens \\ Department of Mathematics\\
Panepistimioupolis, 15784 Athens, Greece
}
\email{kontogar@math.uoa.gr}

\begin{document}
\bibliographystyle{amsplain}

\begin{abstract}
	We study $p$-group Galois covers $X \rightarrow \mathbb{P}^1$ with only one fully ramified
	point in characteristic $p>0$. These covers are important because of the Harbater--Katz--Gabber compactification theorem of Galois
	actions on complete local rings. The sequence of ramification jumps is related to the
	Weierstrass semigroup of the global cover  at  the stabilized point. We determine explicitly the jumps of the  ramification filtrations
	in terms of pole numbers. We give applications
	for curves with zero $p$-rank: we focus on  curves that admit a big action.
	Moreover we initiate the study of the Galois module structure of polydifferentials.
\end{abstract}
\thanks{{\bf keywords:} Automorphisms, Curves, Numerical Semigroups,
	Harbater-Katz-Gabber covers, zero $p$--rank, Big Actions,  Galois module structure.
The first author was supported by a Dahlem Research School and Marie Curie  Cofund fellowship;
he is also a member of the SFB 647 project: Space--Time--Matter, Analytic and Geometric Structures.\\
The second author was supported  by the European Union (European Social Fund - ESF) and Greek national
funds through the Operational Program``Education and Lifelong Learning'' of the National Strategic Reference
Framework (NSRF) - Research Funding Program: THALIS}
\maketitle
\section{Introduction}\label{sec:intro}

Let $X$ be a projective nonsingular algebraic curve of genus $g_X\geq 2$ defined over an {algebraically closed} 
field $k$
of characteristic  $p$. For technical reasons we will exclude the characteristics $p=2,3$ from our study. We will denote by $F$ the function field of the curve $X$ and by  $G$  a subgroup of the automorphism group $\A(X)$.

In the literature there is a lot of interest concerning: the properties of the automorphism group, \cite{StiI, StiII, SalvadorBook, hirschfeld2008algebraic}, its size related to several topological invariants of the curve $X$,
\cite{Giulietti2010-rd, Hu:1893, CKK, Ma-Ro, Nak, Rocher2009-jk}, the deformation theory of the couple $(X,G)$ \cite{Be-Me}, lifting problems \cite{CGH08, CGH11, Obus2011-nl, ObusWewers, Pop14}.

An important tool in understanding the automorphism group is the localization of the action by considering the inertia group $G(P)=\{\sigma \in G: \sigma(P)=P\}$, acting on the local ring $\mathcal{O}_P$ at a $k$-rational point $P$. It is well known, see section \ref{irf}, that the group $G(P)$ admits the
following ramification filtration, 
\[ 
	G(P)=G_{-1}(P) { \supseteq}G_0(P)   \supseteq G_1(P) \supseteq G_2(P) \supseteq \cdots \supseteq \{\mathrm{id}\}.
\]
The determination of the ramification filtration, and its \textit{jumps}, i.e.
the indices such that $G_i(P) \gneqq G_{i+1}(P)$ is  a deep problem. For instance if $G_1(P)$ is abelian, then the Hasse--Arf
theorem \cite[Theorem p. 76]{SeL}  puts very strong divisibility relations among the jumps. 
These jumps appear very often
in a variety of cases in the literature in the local and global function fields, especially when one considers arithmetic problems in algebraic function fields; the most immediate application is the computation of the degree of the 
different of a Galois extension of function fields using  
 Hilbert's {celebrated} formula, \cite[chap. III.4]{StiBo}. 

The knowledge of the jumps is  crucial for expressing obstructions to the lifting problem, see \cite{CGH08, CGH11} and is also related to the Artin representation \cite[Chap. VI]{SeL}. 
For local applications 
and their relation with the famous Hasse--Arf theorem see \cite[Chap. IV]{SeL}; for an unexpected application to normal basis generators, the computation of normal bases 
 is another active research area concerning local and finite fields,
 see the work in \cite{Byott, Lara} and the references therein.

In contrast to their significance, we  only know very few things about 
jumps. More precisely, as far as we know,
  they have been computed  explicitly in function fields,  only for some specific cases, see \cite[Examples 1 to 4]{KontoZ} for a non exhaustive list, and for the more general cyclic $p^n$ case see  \cite[lemma 1]{Valentini-MadanZ}. For the  abelian case $\mathbb{Z}/p^i\mathbb{Z}\times \mathbb{Z}/p^i\mathbb{Z}$, they have been only computed for $i=1,2$; for $i=1$ see \cite[section 3]{Sti:09}, \cite[Theorem 3.11]{QuiRe} while for $i=2$ see \cite{kumar}.

Another direction towards  understanding the  automorphism group $G$, is to consider the representation theory of $G$, acting on several naturally defined vector spaces. A natural choice for vector spaces acted on by the automorphism group,  are the  spaces 
\begin{enumerate}
\item[(i)] $H^0(X,\Omega_X^{\otimes m})$, $m\in \mathbb{N}$,  of holomorphic polydifferentials of $X$, and 
\item[(ii)] Riemann-Roch spaces $L(D)$, for some $G$-invariant divisor $D$.
\end{enumerate}
Concerning the first case, the determination of the Galois module structure 
is an interesting problem 
which has been solved in the following cases:
 for unramified Galois covers \cite{Tamagawa:51, Val:82}; for the {semisimple} part, with respect to the Cartier operator, of $H^0(X,\Omega_X)$ for $p$-covers \cite{Salvador:00, Nak:85}; for cyclic and certain elementary abelian covers, \cite{vm, borne06} and \cite{csm} respectively.
 The main tool that was used in some of the  above references and in our work in this article as well, is the construction of an appropriate basis for the holomorphic polydifferentials. This has also rich connections with other subjects in the literature:
the computation of $n$--Weierstrass points, \cite[Theorem 14.2.48]{SalvadorBook}, \cite{Boseck, garciaa-s, garciaelab, WpFF}; the computation of 
the rank of the Hasse--Witt matrix \cite{Madden78}; the classification of curves with  Hasse--Witt matrix of certain rank \cite{CaSa} and the study of the Artin--Schreier (sub)extensions of the rational function field \cite{Val-Mad:80}.

Conserning the second case, when $D=P$ we can define an action of $G(P)$ on the spaces $L(nP)$ for $n\in \mathbb{N}$.
One can ask if there is any 
relationship between
the localized action of $G(P)$ on $\mathcal{O}_P$ and the natural linear representation on the spaces $\mathrm{GL}(L(nP))$. A way to answer this question is by considering the flag of vector spaces $L(nP)$ for $n\in \mathbb{N}$. 
The possible jumps of the dimension sequence of this natural flag lead to the notion
of pole numbers and Weierstrass semigroups, see definition \ref{pole-defi}.

More precisely, the Weierstrass semigroup $H(P) \subset \mathbb{N}$ is a numerical semigroup consisting of all $n\in \mathbb{N}$ such that there is a function $f$ in the function field of the curve $X$ with pole divisor $(f)_\infty=n P$. We will say that the numerical semigroup $H(P)$ has generators $d_1,\ldots, d_r$ if
\[
H(P)= \mathbb{Z}_+ d_1 + \cdots +  \mathbb{Z}_+ d_r,
\]
where $\mathbb{Z}_+:=\{d\in \mathbb{Z}: d \geq 1\}$. 
Each semigroup has a natural partial ordering: for two elements $a$ and $b$ in the semigroup we  say that $a$ is smaller than $b$ if $b=a+c$ for another element $c$ in the semigroup.  The set of minimal {elements} with respect to  this ordering is called a minimal set of generators for the semigroup, see \cite{FGH}.

An extreme example in the theory of numerical semigroups are the symmetric ones. If we limit ourselves
to the Weierstrass semigroups, then symmetric means that the maximum gap equals the largest possible value: $2g_X-1$. Equivalently, see also \cite[eq. (1.1)]{Gorenstein}, this symmetry is expressed by the following rule:
\begin{equation*} 
x\in H(P) \textrm{ if and  only if } 2g_X-1 -x \notin H(P).
\end{equation*}
Symmetric numerical semigroups are closely connected to the geometry of the curve,
see \cite[Section 7.2]{fp}. Moreover  every such semigroup is the Weierstrass semigroup of a  Gorenstein
curve \cite{Gorenstein}. For an introduction to numerical semigroups, and the importance of the symmetric condition we refer to \cite{FGH, NumSep}.

It is known, see proposition \ref{gap-comp},  that the gaps of the ramification filtration of $G(P)$ are related to the semigroup $H(P)$, since
if $G_{i}(P)  >  G_{i+1}(P)$, for $i\geq 1$, then   $ i=m_r-m_\nu$, for some pole number $m_\nu$,
when $m_r$ is the  smallest pole number at $P$ not divisible by the characteristic $p$.

One of our motivations for this study was to find the set of pole numbers which correspond to jumps of the ramification filtration.

In this note, we have set the following aims:
\begin{enumerate}
\item[(I)]
For ``Harbater-Katz-Gabber covers'', or simply HKG-covers, i.e. Galois covers of the projective line with a unique wildly and at most one tamely ramified point, we will characterize exactly the lower ramification jumps in terms of pole numbers at their unique wildly ramified point and give a complete description for its symmetric Weierstrass semigroup.  We remark that we have not made any assumption for $G_1(P)$ to be an abelian group.
\item[(II)]  We will initiate the study of the Galois module structure of spaces of polydifferentials for HKG-covers and give a basis for their $m$-holomorphic polydifferentials. We will prove that HKG-covers arise in a natural way as Galois covers of curves with zero $p$-rank and apply these results to curves equipped with a ``big action'',  showing also that the module of their holomorphic polydifferentials is an indecomposable $k\left[G_1(P)\right]$-module.
\end{enumerate}

\begin{remark}\label{rem:assumptions}
We focus on the jumps of the ramification filtration. The ramification filtration might jump at $-1$, and in this case  $G_{-1}(P)/G_{0}(P)$ is a nontrivial group isomorphic to the Galois group of the corresponding residue field extension. 
Moreover, we might have a jump at  $0$ if and only if there is tame ramification, since $G_{0}(P)/G_{1}(P)$ equals the tame ramification degree. The crucial information regarding all the other higher order ramification jumps, lie on the  $p$-part of $G$; this is the reason why we assume that our field is algebraically closed and we restrict ourselves to the $p$-part of the ramification filtration, i.e. to $G_{1}(P)$. Thus $G_{-1}(P)= G_0(P)=G_1(P)$. 

Although it seems possible to extend all of our results  over perfect,  instead of  algebraically closed base fields,  there are certain places that have to be treated with some extra attention. 
In the following proofs, we have used explicitly the  fact that  $k$ is an algebraically closed field: in the proof of proposition \ref{action}, and in the proofs of theorem \ref{jumps} \& corollary \ref{case1jumps} respectively. In the latter cases we use \cite[Prop III.7.10]{StiBo}, which requires certain polynomials to have all their roots in $k$.
\end{remark}

It is clear that the group $G_1(P)$ acts on the vector spaces $L(m_iP)$
for each $i\in \mathbb{N}$, defining representations
\begin{equation} \label{rhoidef}
\rho_i: G_1(P) \rightarrow \mathrm{GL}\big(L(m_iP)\big).	
\end{equation}
The second author proved that all, but a finite number of these representations are faithful.
\begin{proposition} \cite[Lemmata  2.1, 2.2]{KontoZ} \label{fai-rep-lemma}
	If $g_X\geq 2$ and $p\neq 2,3$, then there is at least
	one pole number $m_r \leq 2g_X-1$ not divisible by the characteristic $p$.
	Then there is  a faithful representation
	\[
		\rho: G_1(P) \rightarrow \mathrm{GL} \big( L(m_rP) \big),
	\]
	where $m_r$ is the smallest pole number not divisible by the characteristic.
\end{proposition}

\begin{remark}\label{rem:non-faith-tame}
Observe that for a general decomposition group with tame ramification, the above defined representation might not be faithful.
\end{remark}

The above proposition \ref{fai-rep-lemma} is the starting point for defining a new filtration of $G_1(P)$, which we will call the ramification filtration. More precisely, the $i$-th group is just the kernel of the linear representation $\rho_i$ defined in eq. (\ref{rhoidef}).
We refer to section \ref{imr}, for a more detailed  definition. The set of jumps of the representation filtration,
are easier to understand since their definition  is based on representations of the general linear group.
Our study of the jumps of the ramification filtration, with the aid of the representation filtration and the Weierstrass semigroup theory, gives a complete description for {them}.

In general the  ramification filtration can be introduced and studied in terms of general local rings, see \cite{SeL}. In  the case of spectra $\O$ of local rings of the form $k[[t]]$
acted on by a group $G_0$, where $k$ is a perfect 
field of characteristic $p>0$, we can pass from the local case to the global one with the Harbater-Katz-Gabber covers, see definition \ref{HKG-cover-def}. These covers can be seen as a minimal compactification of a local action and there is a lot of interest in them, for instance  they appear in deformation and lifting problems \cite{BCPS, CGH08, CGH11}. 

By considering the Harbater--Katz--Gabber compactification of an action on the
local ring $k[[t]]$, we have the advantage of being able to attach global invariants,
like genus, $p$-rank, differentials etc. to the local case.
Also finite subgroups of the  automorphism
group $\A k[[t]]$, a subject difficult to understand, but crucial for studying deformation theory
of curves with automorphisms \cite{Be-Me},
become  subgroups of $\mathrm{GL}(V)$ for a finite dimensional vector space $V$.






We would like to point out that in the case of Riemann surfaces such a relation among the group $G(P)$ and the Weierstrass
semigroup at $P$ is known.  I. Morisson and H. Pinkham \cite{MorrisonPinkham} studied this connection in
characteristic zero  for \textit{Galois Weierstrass points}: a point $P$ on a compact Riemann surface $Y$ is called Galois Weierstrass if
for  a meromorphic
function $f$  on $Y$ such that $(f)_\infty=dP$, where $d$ is the least pole number
in the Weierstrass semigroup at $P$, the function $f:Y\rightarrow \mathbb{P}^1(\mathbb{C})$ gives rise to
a Galois cover.  This article can be seen as a natural
generalization of some results in that article in positive characteristic. Notice that {in the latter case,}
the first non zero element in $H(P)$ is not enough to grasp the group structure.
We have to go up to the first pole number in $H(P)$ that is not divisible by $p$ to do so.
And of course the stabilizer $G(P)$  and its $p$-part $G_1(P)$  do not have to be cyclic groups anymore.


Our motivation for studying actions on HKG-covers was the deformation theory of curves with automorphisms.
J. Bertin and A. M\'ezard in \cite{Be-Me} proved  a  local global principle that can be used to show that
the ``difficult part'' of the study of the deformation functor of curves with automorphisms resides
in the local deformation functors. 

This is a too vast object of study to describe here; the reader is advised
to look
at \cite{Be-Me} for more information. Local actions can be compactified to HKG-covers,
and at least the dimension of the tangent space of the deformation functor is reflected by the
space of $2$-holomorphic differentials $H^0(X,\Omega_X^{\otimes 2})$
of the corresponding HKG-cover. Indeed, the second
author in  \cite{KoJPAA06} related the dimension  of the space of coinvariants of global sections of 2-polydifferentials $\dim  H^0(X,\Omega_X^{\otimes 2}) _G$,  
to the dimension of the tangent space of the deformation functor of curves with automorphisms. This computation is  a complicated task and reserves further study. 



The structure of the article is as follows:
In section \ref{DefinitionMR} we review some basic notions for the ramification filtration, see section \ref{irf}, and the Weierstrass semigroup
at a fixed point of our curve $X$, see section \ref{iws}. After that, we see how these two notions are related in section \ref{irrs} and focus 
on the HKG-covers, where the Galois group is not necessarily an abelian group. We
finally define the representation filtration and give all the necessary background in order to state our two main results in section \ref{imr}. Section \ref{section-2} provides some information concerning the Weierstrass semigroup at a totally ramified point for a general Galois cover. Section \ref{mainsection} is the heart of this note providing the proofs for the computation of the ramification jumps (in upper and lower numbering). Section \ref{apply} is devoted to applications of our main results: we mainly focus on curves with big-actions, see  section \ref{bigactionsap}.  
These curves, like any other HKG-cover, are curves with zero $p$-rank,
 see section \ref{zeroprankapp}.
Finally in section \ref{HasseArfapp} we interpret the Hasse-Arf theorem  in terms of our results. 
In section \ref{represent-KG} we provide a basis for holomorphic polydifferentials. This will characterize all the Weiertrass semigroups that we have previously computed as symmetric; on the other hand this will also be the starting point for studying the Galois module structure.

\subsubsection*{Acknowledgement:} 
We would like  to thank the anonymous referee for many useful comments.  
The current form of the article owes a great deal to the referee's constructive criticism and detailed suggestions.

\section{Definitions and main results}
\label{DefinitionMR}
\subsection{Ramification filtration}\label{irf}
%
%
%
%
Let $\O_P$ be the completed local ring at the $k$-rational point  $P$ and let $m_P$ be its maximal ideal.
 The subgroup $G_i(P)\subset G(P)$ is defined
as the subgroup of $\sigma \in G(P)$ which act trivially on $\mathcal{O}_{P}/m_{P}^{i+1}$.
The groups $G_{i}(P)$ form a filtration:
\[ 
	G(P)=G_{-1}(P){ \supseteq}G_0(P)   \supseteq G_1(P) \supseteq G_2(P) \supseteq \cdots \supseteq \{\mathrm{id}\}.
\]
It is known that $G_1(P)$ is the $p$-part of $G(P)$, $G_0(P)/G_1(P)$ is a cyclic group of order prime to $p$, while for $i\geq 1$, the quotients $G_{i}/G_{i+1}$ are elementary abelian groups. 
The quotient $G(P)/G_{0}(P)$ is isomorphic to the Galois group
$\mathrm{Gal}\left(\frac{\mathcal{O}_P}{P}/\frac{\mathcal{O}^G_P}{\mathcal{O}^G_P \cap P}\right)$. This latter group is trivial if $k$ is algebraically closed. By remark \ref{rem:assumptions} we will restrict ourselves to the study of the jumps of the $p$-part $G_{1}(P)$.

Let us fix the notation for the jumps:
\begin{equation}\label{ramfiltration}
	G_0(P) =G_1(P)=G_{b_1}>G_{b_2}>\cdots > G_{b_\mu}>\{\textrm{id} \}.
\end{equation}
This means that $G_{b_\nu} \varsupsetneq G_{b_\nu+1}=G_{b_{\nu+1}}$ for every
$1 \leq \nu \leq \mu$ and that there are $\mu$ jumps.

The theory of ramification filtrations can be considered more generally for complete discrete valuation rings see \cite[chap. IV]{SeL}. We will see in section \ref{section-HKG}, that such local actions on rings $k[[t]]$ can always come from actions on curves. 


\subsection{Weierstrass semigroups}\label{iws}

 Consider the flag of vector spaces
\[ 
	k=L(0)=L(P)= \cdots  =L((i-1)P) < L( iP) \leq  \cdots \leq  L( (2g_X-1)P ),
\]
where
\[
	L(iP):=\{ f \in F: {\rm div}(f) +iP \geq 0 \} \cup \{0\}.
\]
We will write $\ell( D)= \dim_k L(D)$, for a divisor $D$.
\begin{definition} \label{pole-defi}
	An integer $i$ will be called  a pole number if there is a function
	$f \in F^*$ so that $(f)_\infty=iP$ or
	equivalently $\ell\big( (i-1)P \big)+1=\ell\big(i P\big)$. If $i$ is not a pole number, we call it a gap.
	The set of pole numbers at $P$  form a numerical semigroup $H(P)$ which is called the
	Weierstrass semigroup at $P$.
\end{definition}
Note that $0$ is always a pole number; thus from now on when
we write $H(P)$ we always assume that $\{0\}\in H(P)$ for every Weierstrass semigroup.
It is  known that there are exactly $g_X$ pole numbers that are smaller than or equal to $2g_X-1$ and that every
integer $i \geq 2g_X$  is in the Weierstrass semigroup, see \cite[I.6.7]{StiBo}. 

%
%
\subsection{Action on Riemann-Roch spaces}\label{irrs}

\begin{definition}\label{fi}Let $m_r$ be the smallest pole number at $P$ not divisible by $p$. Denote
	by \[0=m_0< \cdots < m_{r-1} < m_r \]    all the pole numbers at $P$ in increasing sequence which  are
	$ \leq m_r$.
	From now on, $f_i \in F$, {for} $0\leq i \leq r$, will denote a selection of a function such that $(f_i)_\infty=m_iP$.
\end{definition}

\begin{remark}\label{nonuniq}
	Observe that a function which  has a unique pole at $P$ of  order $m_i$ is not unique.
	If $f_i,f_i'$ are two functions such that
	$(f_i)_\infty=(f_i')_\infty=m_iP$, then by examining the Laurent expansion
	of $f_i,f_i'$, there is constant $C\in k^*$ such that:
	\[
		f_i'=C f_i + g,
	\]
	where $g$ is a function in $L(m_{i-1} P)$.
\end{remark}

Concerning the jumps of the ramification filtration we have the following characterization:
\begin{proposition} \cite[prop. 2.3]{KontoZ}\label{gap-comp}
	Let $X$ be a curve acted on by the group $G$. For every fixed point $P$ on $X$ we consider  the corresponding  faithful
	representation  defined  in  proposition \ref{fai-rep-lemma}:
	\[\rho: G_1(P) \rightarrow \mathrm{GL}_{\ell(m_r P)} (k).\]
	If $G_{i}(P)  >  G_{i+1}(P)$, for $i\geq 1$, then   $ i=m_r-m_\nu$,
	for some pole number $m_\nu$.
\end{proposition}
Since we characterize exactly the jumps and the structure of the Weierstrass semigroups at the unique ramified point of a HKG-cover, we also characterize exactly the set of pole numbers $m_{\nu}$ for $\nu< r$, such that $m_{r}-m_{\nu}$ is a jump.
\subsection{Representation Filtration}\label{imr}
{Recall that \[0=m_0< \cdots < m_{r-1} < m_r \]  are all the pole numbers at $P$ in increasing sequence up to $m_r$.}
\begin{definition}\label{repfil}
	For each $ 0 \leq i \leq r$ we consider the representations
	\[
		\rho_i :G_1(P) \rightarrow \mathrm{GL}(L(m_i P)).
	\]
	We form the decreasing sequence of groups:
	\begin{equation} \label{rep-filt}
		G_1(P)=\mathrm{ker} \rho_0 \supseteq \mathrm{ker} \rho_1 \supseteq \mathrm{ker} \rho_2
		\supseteq \cdots \supseteq
		\mathrm{ker} \rho_r=\{1\}.
	\end{equation}
	We will call this sequence of groups {\em the representation filtration}.
\end{definition}
\begin{remark}
The $i$-th ramification group is the kernel of the map:
\[
   \phi_i: G_0(P) \rightarrow \mathrm{Aut}\left(\O_p/ m_P^{i+1} \right)
\]
while the $i$-th representation group is the kernel of the map 
\[
   \rho_i: G_0(P) \rightarrow \mathrm{GL}\big( L(m_i P) \big),
\]
where $L(m_i P)$ can also be seen as a quotient of $L(m_r P)=\Big(L(m_i P) \oplus W \Big)/W$, where $W$ is the vector space complement of $L(m_i P)$ in $L(m_rP)$. 
\end{remark}

{ Note that the spaces $L(m_i P)$ are fixed by the action of $G_1(P)$}. The filtration of {eq. (\ref{rep-filt})} leads to a successive sequence of elementary abelian $p$-group extensions of the field $F^{G_1(P)}$:
\begin{equation} \label{F-seq}
	F^{G_1(P)}=F^{\mathrm{ker} \rho_0} \subseteq F^{\mathrm{ker} \rho_1} \subseteq F^{\mathrm{ker} \rho_2}
	\subseteq \cdots \subseteq
	F^{\mathrm{ker} \rho_r}=F.
\end{equation}
We call an index $i$ a jump of the representation filtration if and only if $\ker \rho_i\gneq \ker \rho_{i+1}$.
Let us also fix the notation for the representation jumps:
\[
	G_1(P)=\ker \rho_0 = \cdots= \ker \rho_{c_1} > \cdots >\ker \rho_{c_{n-1}}>\ker \rho_{c_{n}}>\{\textrm{id} \}.
\]
In other words, the above sequence jumps at n integers.
These integers will be called  jumps of the representation
filtration,
\begin{equation} \label{rep-jumps}
  c_1<c_2< \cdots < c_{n-1} < c_{n}=r-1.
\end{equation}
The last equality $c_{n}=r-1$  comes from the faithful representation of proposition \ref{fai-rep-lemma}, since $\ker \rho_r=\{1\}$, coupled with 
 lemma \ref{semigroupup} {which will be proved later}.  Notice that $c_i \in \{1, \ldots, r\}$ for all $1\leq i \leq n$.
\begin{remark}\label{structureker}
	Every element $\sigma\in \mathrm{ker}{\rho_i}$ fixes by definition all $f_\nu$ such that $(f_\nu)_\infty=m_\nu P$
	for $\nu\leq i$. A nonnegative integer $i$ is a jump whenever the function $f_{i+1}$ is not  $\ker \rho_i$ invariant.
\end{remark}

{We will prove in proposition \ref{jumpsrho-gen} that if}  $c_i$ is a representation jump then $m_{c_i+1}$ is a minimal generator of $H(P)$. { At every jump of the sequence of the } groups $\ker \rho_{c_i}$, the corresponding sequence of fields will
also jump and moreover
\begin{equation}\label{N-1}
	F^{\ker \rho_{c_{i+1}}} = F^{\ker \rho_{c_{i}}}(f_{c_i+1}).
\end{equation}
\begin{definition}
	In order to simplify notation we set $F_i=F^{\ker{\rho_{c_i}}}$, $\bar{m}_i=m_{c_i+1}$ and $\bar{f}_i=f_{c_i+1}$. Denote also by $p^{h_i}=|\ker \rho_{c_{i+1}}|$,  for all $1 \leq i \leq n-1$,
	and $p^{h_0}=G_1(P)$.
\end{definition}
Thus eq. (\ref{N-1}) can be written as
\[F_{i+1}=F_i(\bar{f}_i).\]
{We will prove in lemma \ref{semigroupup} that}
 in every extension we
add an extra function $f_{c_i+1}=\bar{f_i}$. 
Define
$Q_i=F_i \cap P$ for $1 \leq i \leq n+1$ to be the \textit{unique} ramification point of the tower defined in eq. (\ref{F-seq}). At the level of the Weierstrass semigroups, the field generator $\bar{f}_i$ adds a new generator $\bar{m}_i$ in the image of the 
 {semigroup} $H(Q_i)$ on $H(Q_{i+1})$.  In section \ref{section-2} in lemma \ref{up-down}, we will see how the Weierstrass semigroups at the ramified points of a Galois extension of fields are related.
\maketitle
Using
this relation,  
the semigroup of  $F_2$ at $Q_{2}$  is
\[H(Q_2)=\left| \frac{\mathrm{ker} \rho_{c_1}}{\mathrm{ker} \rho_{c_2}} \right| \Z_+ +  \lambda_1 \Z_+ = { p^{h_0-h_{1}}} \Z_+ + \lambda_1 \Z_+\]
with
$(\lambda_1,p)=1$.

Notice that $\lambda_1=1$ if and only if $F^{\mathrm{ker} \rho_{c_2}}$  is rational. We proceed in this way to obtain
\[
	H(Q_{i+1})={ p^{h_{i-1}-h_{i}}}H(Q_{i}) + \lambda_i \Z_+, \textrm{ for all } 1 \leq i \leq n,
\]
where $(\lambda_i,p)=1$.
We will see in proposition \ref{jumpsrho-gen}
that the elements
\[
	p^{h_1} \lambda_1 <   p^{h_2}\lambda_2 < \cdots < p^{h_{n-1}}\lambda_{n-1}<  \lambda_{n}=\frac{m_{c_n+1}}{|\ker \rho_{c_{n+1}}|}=m_r,
\]
are inside the set of  generators of the Weierstrass semigroup  at $P$. If we add the element
$p^{h_0}$ then, proposition \ref{weiersemi} will give:
\[
	\langle p^{h_0},p^{h_1}\lambda_1, \ldots,p^{h_{n-1}}\lambda_{n-1},\lambda_{n}  \rangle_{\mathbb{Z}_+} = H(P).
\]
We have the following picture of fields, groups, places and semigroups:
\[
	\xymatrix{
		F \ar@{-}[d] & \{1 \} \ar@{-}[d] & P \ar@{-}[d]   & H(P) \ar@{-}[d]
		\\
		F_{i+1}=F^{\ker \rho_{c_{i+1}}} \ar@{-}[d] & {\ker \rho_{c_{i+1}}} \ar@{-}[d] & Q_{i+1} \ar@{-}[d]  &  H(Q_{i+1}) =
		\left\langle { p^{h_{i-1}-h_{i}}} H(Q_{i}),    {\lambda_i } \right\rangle_{\mathbb{Z}_+} \ar@{-}[d]
		\\
		F_{i}=F^{\ker \rho_{c_{i}}} \ar@{-}[d] & \ker \rho_{c_{i}}  \ar@{-}[d]  & Q_{i} \ar@{-}[d]  &  H(Q_{i})     \ar@{-}[d]
		\\
		F_0=F^{G_1(P)} & G_1(P)  &  {Q_1}   &  \mathbb{Z}_+
	}
\]


%
%
%
\subsection{Harbater-Katz-Gabber covers}
\label{section-HKG}

\begin{definition}\label{HKG-cover-def}
	A Harbater-Katz-Gabber cover is a Galois  cover $X_{\mathrm{HGK}}\rightarrow \mathbb{P}^1$, such that there are at most two $k$-rational points $p_{1},p_{2} \in \mathbb{P}^{1}$, such that $p_{1}$ is tamely ramified and $p_{2}$ is fully ramified. All other geometric points of $\mathbb{P}^{1}$ remain unramified. In this article we are interested in  $p$-groups, so for us  HKG-covers have a unique ramified point.
\end{definition}


So far we have started with a subgroup of $\A (X)$ that is the isotropy group $G(P)$ of a fixed point $P$ of $X$. On the other hand, in section \ref{apply} we will see in theorem \ref{zeroprankGK}, that $X_{\mathrm{HKG}}$ has zero $p$-rank and thus every $p$-subgroup $G$ of $\A(X_{\mathrm{HKG}})$ can be realized as the stabilizer of a unique point $P$, see \cite[paragraph 11.13]{hirschfeld2008algebraic}, thus $G=G(P)$. 



The Harbater--Katz--Gabber compactification theorem \cite{Harbater:80}, \cite[th. 1.4.1]{KaGa}, {for the case of $p$-groups} asserts that there is a HKG-cover
$X_{\mathrm{HKG}}\rightarrow \mathbb{P}^1$ ramified only at one point $P$, with Galois group
$G=\mathrm{Gal}(X_{\mathrm{HKG}}/\mathbb{P}^1)=G_0$ such that $G_0(P)=G_0$, and the action of
$G_0$ on the completed local ring $\hat{\O}_{X_{\mathrm{HKG}},P}$ coincides with the original
action of $G_0$ on $\O$.

For the case of HKG-covers, we will show in corollary \ref{case1jumps} that the subset of minimal generators $\bar{m}_{1},\ldots,\bar{m}_n$ of the Weierstrass semigroup described in  Proposition \ref{jumpsrho-gen}, is the whole set of minimal generators unless $G_{1}(P)=G_{2}(P)$. In the latter case, we will prove in proposition \ref{gen-chara}, that we have also to add  $|G_{1}(P)|$, 
in order to obtain the full set of minimal generators of the semigroup. 
In proposition \ref{action} we will describe the action of the Galois group on the generators of the tower of the fixed fields by the kernels; this will be a fundamental step for the computation of the jumps that will be given in theorem \ref{jumps}.
We will also prove in corollary \ref{sameorder}, that the representation and ramification filtrations coincide. By these two results, the jumps of the ramification filtration are completely determined.
A basis of holomorphic polydifferentials will be given in proposition \ref{deg-char}; this will help us to derive some useful information  for their Galois module structure in proposition \ref{KGmodule}.
Finally, this basis of holomorphic polydifferentials also proves in corollary \ref{symmetricWeier}, that the Weierstrass semigroup at the ramified point is symmetric.

\subsection{Main results}

Now we are ready to state our two main theorems:
{Concerning the structure of  $H(P)$, the Weierstrass semigroup at $P$,  we have the following}
\begin{theorem}\label{mainresweier}
	\begin{enumerate}
		\item\label{b} For every jump of the representation filtration
		$c_i$, $1 \leq i \leq n$ there exists a generator of $H(P)$ of the form  $\bar{m}_i=m_{c_i+1}=p^{h_i} \lambda_i$,
		where $(\lambda_i,p)=1$.
		\item\label{d} {The first ramification jump affects the structure of $H(P)$ in the following way:}
		\begin{enumerate}
			\item\label{d1} If $G_1(P)> G_2(P)$, then  the extension $F/F^{G_2(P)}$ is also HKG,
			and the Weierstrass semigroup $H(P)$ is minimally generated by ${\bar{m}_i}$, with $1\leq i \leq n$.
			Moreover $|G_2(P)|={\bar{m}_1}=m_1$.
			\item\label{d2} If $G_1(P)=G_2(P)$ then we need $\bar{m}_i$,  $1 \leq i \leq n$ together with $p^{h_0}=|G_1(P)|$  in order to generate $H(P)$.
			In this case $|G_1(P)| \neq \bar{m}_{i}$ for all $1 \leq i \leq n$.
		\end{enumerate}
		In both cases the semigroup  $H(P)$ is symmetric.
	\end{enumerate}
\end{theorem}

\begin{proof}
	{Part (\ref{b}) will be proved in  proposition \ref{jumpsrho-gen};  Part \ref{d1}  will be proved in  corollary \ref{case1jumps}, while part  
	 (\ref{d2}) will be proved in  proposition \ref{gen-chara} and lemma \ref{weirgenbutnotrepjump}. Finally, the assertion about the symmetric Weierstrass semigroup will follow  from corollary \ref{symmetricWeier}.}
\end{proof}


The relationship between the representation and the ramification filtrations is given in terms of the following:

\begin{theorem} \label{mainres}
	Assume that $X \rightarrow X/{G_1(P)} =\mathbb{P}^1$  is a HKG-cover.
	Then
	\begin{enumerate}
		\item\label{a} The jumps of the ramification filtration are the integers $\lambda_i$ for  $1\leq i\leq n$, i.e.  $\lambda_{i}=b_i$ for
		every such $i$, while the number of ramification and representation jumps coincide, i.e. $\mu=n$.
		\item\label{c} $G_{b_i}=\ker \rho_{c_{i}}$  for all   $2\leq i \leq n$.
	\end{enumerate}
\end{theorem}

\begin{proof}
	{ Part (\ref{a}) will be proved in  theorem \ref{jumps}, while 
	(\ref{c}) will be proved in  corollary \ref{sameorder}. }
\end{proof}

\begin{remark}In view of remark \ref{rem:assumptions}, let us assume 
that we have a HKG-cover which is also tamely ramified.
Since $G_1(P) \lhd G_0(P)$ we have the following picture of curves, function fields and ramified places:
\[
\xymatrix{
	X \ar[d]  & F \ar@{-}[d] & P \ar@{-}[d] & \\
	\mathbb{P}^1 \cong X/G_1(P) \ar^{C_n}[d] & F^{G_1(P)} \ar@{-}[d]  
	& Q \ar@{-}[d] & Q'  \ar@{-}[d] \\
	\mathbb{P}^1 \cong X/G_0(P)  &  F^{G_0(P)} & q & q' 
}
\]
Keep in mind that $G_0(P)$ is the semisimple product of the cyclic $G_0(P)/G_1(P)$ and $G_1(P)$. Now, the lower ramification jumps are at 0, while the rest of them are given by theorem \ref{mainres}. 
\end{remark}

\section{ Totally ramified Galois covers }
\label{section-2}
%
%
%
%
%
%
\label{Jumps}
We begin our study by relating the Weierstrass semigroups at totally ramified 
 points 
 of Galois covers over  {algebraically closed} fields in positive characteristic.
We remark that the results obtained in this section are not limited to $p$-groups.
Consider a Galois cover $\pi:X\rightarrow Y=X/G$ of algebraic curves,
and let $P$ be a fully ramified $k$-rational point of $X$. How are the Weierstrass semigroup
sequences of $P$, and $\pi(P)$ related?
\begin{lemma} \label{up-down}
	Let $F(X)$, $F(Y)=F(X)^G$
	denote the function fields of the curves $X$ and
	$Y$ respectively. The morphisms
	\[
		N_G:F(X) \rightarrow F(Y) \mbox{ and } \pi^*:F(Y) \rightarrow F(X),
	\]
	sending $f\in F(X)$ to $N_G(f)=\prod_{\sigma_\in G} \sigma f$  and $g\in F(Y)$
	to $\pi^* g\in F(X)$ respectively, induce
	injections
	\[
		N_G: H(P) \rightarrow H(Q)  \mbox{ and } \pi^*:H(Q) \stackrel{\times |G|}\longrightarrow H(P),
	\]
	where $Q:=\pi (P)$.
\end{lemma}
\begin{proof}
	For every element $f \in F(X)$ such that $(f)_\infty=mP$,
	the element $N_G(f)$ is a $G$-invariant element,  so it is in $F(Y)$.
	Moreover, the pole order of $N_G(f)$ seen as a function on $F(X)$ is
	$|G|\cdot m$. But since $P$ is fully ramified the valuation of $N_G(f)$
	expressed in terms of the local uniformizer at $\pi(P)$ is just $-m$.

	On the other hand an element $g\in F(Y)$ seen as an element
	of $F(X)$ by considering the pullback $\pi^*(g)$ has for the same reason
	valuation at $P$ multiplied by the order of $G$.
\end{proof}

\begin{remark}
	The condition of full ramification is necessary in the above lemma.
	Indeed, if a point $Q \in Y$ has more than one elements in $\pi^{-1}(Q)$
	then the pullback of $g$  is supported
	on $\pi^{-1}(Q)$ and gives no information for the Weierstrass semigroup at
	any of the points $P \in \pi^{-1}(Q)$.
\end{remark}
\begin{corollary}\label{g0}
	The order  $|G|\in H(P)$ if and only if $g_{X/G}=0$.
\end{corollary}
Another immediate consequence of lemma  \ref{up-down} is the following
\begin{corollary} \label{div-inv}
	If an element $f$ such that $(f)_\infty = a P$ is invariant under the
	action of a subgroup  $H < G$, then $|H|$ divides $a$.
\end{corollary}
\begin{proof}
	Since $f$ is invariant it is the pullback of a function $g \in F(X/H)$. The result now follows from lemma \ref{up-down}.
\end{proof}

\section{Enumerating jumps}\label{mainsection}
%
%

Recall that an index $i$ is a jump of the representation filtration if and only if $\ker \rho_i\gneq \ker \rho_{i+1}$ and that we have the following sequence for the representation jumps:
\[
	G_1(P)=\ker \rho_0 = \cdots= \ker \rho_{c_1} > \cdots >\ker \rho_{c_{n-1}}>\ker \rho_{c_{n}}=\ker \rho_{r-1}>\{\textrm{id} \}.
\]
\begin{proposition} \label{jumpsrho-gen}
	If $\mathrm{ker}\rho_{c_i} \varsupsetneq \mathrm{ker} \rho_{c_i+1}$, i.e., when $c_i$ is a representation
	jump then $\bar{m}_i=m_{c_i+1}$ is a minimal generator of $H(P)$. 
\end{proposition}
\begin{proof}
Fix elements $f_i \in L(m_rP)$ with pole numbers $m_i$ respectively.  
	Suppose that $\ker \rho_{c_i} \varsupsetneq \ker\rho_{c_i+1}$. 
	By definition the element $f_{c_i+1}$ is not fixed by $\ker \rho_{c_i}$.
    Observe also that every function in $L(mP)$, with $m<m_{c_i+1}$ is by definition fixed by $\ker \rho_{c_i+1}$.

	If $m_{c_i+1}$ is  in the semigroup $\langle m_1,\ldots, m_{c_i}\rangle_{\mathbb Z_+}$
	generated by all $m_1,\ldots,m_{c_i}$ then
	\begin{equation}\label{writepn}
		m_{c_i+1}=\sum_{j\leq c_i} \nu_j m_j, \textrm{ where } \nu_j \in \mathbb{Z}_+,
	\end{equation}
	and there is a constant $C\in k^*$ such that:
	\begin{equation}\label{writeratf}
		f_{c_i+1}=C\cdot\prod_{j\leq c_i} f_j^{\nu_j} + 
\Lambda_{c_i+1},
\end{equation}
where $\Lambda_{c_i+1}$ is a sum of terms such that the degree of their polar part is smaller than $m_{c_i+1}$.
	But this is impossible since every element 
	$\sigma\in \mathrm{ker}{\rho_{c_i}}$ fixes the right hand side of the last equation, therefore 
	 $\ker \rho_{c_{i+1}}=\ker \rho_{c_i}$, a contradiction.
The reader should notice that,  in general, the expression given in eq. (\ref{writepn}) is not unique. This fact does not affect the proof of the  proposition.
\end{proof}

\begin{remark}\label{fields}
	The fields $F, \; F_i$, $i=1,\ldots,n$ given in eq. (\ref{F-seq}) and in definition \ref{repfil}  are generated by the elements
	$\bar{f}_i$ that we introduced in each step, i.e.
	\[
		F_{i+1}=F_i(\bar{f}_i)=F^{G_1(P)}(\bar{f}_1,\ldots,\bar{f}_i).
	\]
	Moreover $F^{G_1(P)}=k(f_{i_0})$ for some index $i_0$,
%
and
	$
	F=k(f_{i_0},\bar{f}_1,\ldots,\bar{f}_{n}=f_r).
	$
	We form the fields   $F_i$ by successive extensions of the rational function field
	$F^{G_1(P)}$. At every jump  $c_i$ of the representation filtration we add an extra element $\bar{f}_i$ to the field
	$F_i$. 
\end{remark}

%
%
%
%
\subsection{Examples}\label{ex}
%
%
The converse of proposition \ref{jumpsrho-gen} is wrong.
We will give examples of curves where $m_{i+1}$, for some index $i\in \mathbb{N}\cup\{0\}$, is a generator of the Weierstrass semigroup
but $i$ is not a representation jump, i.e.: $\ker \rho_{i} =\ker \rho_{i+1}$. In the first example provided below we can take ${i}=0$.
\begin{example}
	Consider the Artin Schreier extension of the rational function field  given by the equation
	\[
		y^p -y=f(x),
	\]
	where $f(x)$ is a polynomial that has a unique pole at $\infty$ and $\deg f(x)=m_r$, $(p,m_r)=1$. Suppose that $m_r>p$. It is well known that the Weierstrass semigroup at {$P$, the point above $\infty$}, is given by $\langle p, m_r \rangle_{\mathbb{Z}_+}$ \cite[p. 618]{StiII}.
	Notice that $|G|=|G_1(P)|=|\ker \rho_0|=p$, with $m_1=p$ a generator of the Weierstrass semigroup \textit{but} $\ker \rho_0 =\ker \rho_1$ since $|\ker \rho_0|$ divides  $m_1$, so  $f_1$ is a $\ker \rho_0$-invariant element
	and $0$ is not a representation jump. Notice that here $m_r=-v_P(y)= -v_\infty(f(x))$ is the unique ramification jump of $G_1(P)$.
\end{example}

Next we will give an example, namely the Giulietti--Korchm\'aros  curve (see \cite{newfamilyofmaximal}), where
$m_{i +1}$ is  a Weierstrass generator at $P$ with $i \neq 0$ such that $\ker \rho_{i} =\ker \rho_{i+1}$.

\begin{example}[The $GK$--curve]\label{GK}
	Let $\xi=p^\alpha$ for a positive integer $\alpha$ and $q=\xi^3$. Let
	\[
		h(X)=\sum_{\kappa=0}^\xi(-1)^{\kappa+1}X^{\kappa(\xi-1)}.
	\]
	In the three dimensional projective space over $\bar{\mathbb{F}}_{q^2}$,
	the curve $X_{GK}$ that results as the complete intersection of the surface with affine equation
	\[
		Z^{\xi^2-\xi+1}=Yh(X)
	\]
	and the Hermitian cone with affine equation
	\[
		X^\xi+X=Y^{\xi+1}
	\]
	is called the GK curve \cite{newfamilyofmaximal}. It has a unique infinite point $P$, and it is maximal over $\bar{\mathbb{F}}_{q^2}$ \cite[Theorem 1]{newfamilyofmaximal}, i.e. the number of its $\mathbb{F}_{q^2}$-rational points attains the Hasse--Weil  upper bound $q^2+1+2g_{X_{\mathrm{GK}}} q$. This example provides us with
	one of the few known families of curves that are maximal. 
	Note that in \cite{GKgen} a generalization of the above curve is given,
	the so-called generalized GK curve.
	The Weierstrass semigroup {at $P$} is generated by $\langle m_1, m_{2}, m_3 \rangle_{\mathbb{Z}_+}$, with
	$m_1=\xi^3-\xi^2+\xi$, $m_{2}=\xi^3$ and  $m_3=\xi^3+1$ \cite[Proposition 1]{newfamilyofmaximal}. Notice that $m_{2}=\xi^3=|G_1(P)|$, see \cite{stefanigiul,newfamilyofmaximal}
	and $F^{G_1(P)}=k(f_2)$.
	We compute the representation filtration and the picture is the following
	\[
		G_1(P)=\ker\rho_0 \supsetneq \ker \rho_1=\ker \rho_{2} \supsetneq \{\textrm{id}\}.
	\]
	That is $m_2$ is a generator \textit{but} $1$ is not a representation jump (notice also that $|\ker \rho_2|=\xi$).
	Here $F^{\ker \rho_2}=k(f_1, f_2)=F^{G_1(P)}(f_1)$, see \cite{stefanigiul}. Moreover there are two ramification jumps
	for this case, \cite[Proposition 4.2]{stefanigiul}: $m_r=-v_P(f_3)$ and $\frac{m_1}{|\ker \rho_2|}$.
\end{example}
%
%
%
%
\subsection{Structure of the Weierstrass semigroups, Galois action \& computation of ramification jumps}

Recall that  $F_{i}=F^{\ker \rho_{c_i}}$, and  $Q_i:=F_i\cap P$ for $1 \leq i \leq n+1$ is  the restriction of the place $P$ to  the intermediate field $F_i$.
Keep in mind that  $r$ counts the number of elements in the Weierstrass semigroup 
up to the first pole number that is not divisible by $p$,
 while $n$ counts the number of representation jumps.
\begin{lemma} \label{semigroupup}
For all  $1\leq i \leq n$,
	the semigroup  $H(Q_{i+1})$ is generated by elements of the semigroup  $H(Q_{i})$ multiplied by
	$p^{h_{i-1}-h_{i}}=[\ker \rho_{c_{i}} : \ker \rho_{c_{i+1}}]$ and an extra {prime to $p$} minimal generator:
	\[
		-v_{Q_{i+1}}(\bar{f}_i)=\frac{m_{c_{i}+1}}{|\ker\rho_{c_{i+1}}|}=\bar{m}_i p^{-h_i}, 
	\]
	where $c_{i}$ 
	is a representation jump and $\bar{m}_i$ the minimal extra generator for $H(Q_{i+1})$ compared to $H(Q_i)$, by proposition \ref{jumpsrho-gen}.
\end{lemma}
\begin{proof}
	From lemma \ref{up-down} in every step of the representation tower we have
	\[\left|\frac{\ker \rho_{c_{i}}}{ \ker \rho_{c_{i+1}}}\right| H(Q_{i})=p^{h_{i-1}-h_{i}} H(Q_{i})
		\subset  H(Q_{i+1})
		.\]
		We will  apply proposition  \ref{jumpsrho-gen} to the extension $F_{i+1}/F^{G_1(P)}$.
		The group $\ker \rho_{c_{i+1}}$ is a normal subgroup of $G_1(P)$ as kernel of a homomorphism.
		Recall that ${F_{i+1}}=F^{\ker \rho_{c_{i+1}}}$ and notice now that the field  extension ${F_{i+1}}/F^{G_1(P)}$ is also HKG
		and their representation filtration is
		obtained from the quotients of the representation filtration of $F/F^{G_1(P)}$ by the group
		$\ker \rho_{{c_{i+1}}}$. Therefore,
		according to proposition \ref{jumpsrho-gen}  and from basic properties arising from the definition,
		see remarks \ref{structureker} and \ref{fields}, $H(Q_{i+1})$ will have an extra generator compared to $H(Q_{i})$, which is coming from the generator of the extension
		$ F_{i+1} / F_{i}$ which is $\bar{f}_i$.
		Using lemma \ref{up-down} we have
		\begin{equation} \label{bas-val1}
		-v_{Q_{i+1}}(\bar{f}_i)=\frac{\bar{m}_i}{|\ker\rho_{c_{i+1}}|}.	
		\end{equation}
		We will now prove that $\bar{f}_i$ has prime to $p$ pole order. We know by proposition
		\ref{fai-rep-lemma} that there is a  prime to $p$ pole number $m$ { minimally chosen} in $H(Q_{i+1})$ together with an element
		$g$ such that $(g)_\infty=m Q_{i+1}$,
		and the action {$\rho_m$} of $\mathrm{Gal}(F_{i+1}/F_i)$ on  $L(mQ_{i+1})$ is faithful. This proves that $g$ generates $F_{i+1}$ over $F_{i}$. Indeed, if this was not the case, then $\{\textrm{id}\}\neq\mathrm{Gal}\left(F_{i+1}/F_i(g)\right)\subseteq\ker\rho_m=\{\textrm{id}\}$.
		It is clear that
		$\frac{\bar{m}_i}{|\ker\rho_{c_{i+1}}|} \leq m$, since  $\frac{\bar{m}_i}{|\ker\rho_{c_{i+1}}|}$ is the smallest element in $H(Q_{i+1})$ not in
		$p^{h_{i-1}-h_{i}}H(Q_i)$, {note also that  if  $\frac{\bar{m}_i}{|\ker\rho_{c_{i+1}}|} > m$ then $g$ would be, by construction, $\ker\rho_{c_i}$-invariant}.

		Every element in the semigroup $H(Q_{i+1})$ should be the pole number of a polynomial in
		$k[\bar{f}_0,\ldots,\bar{f}_{i-1},g]$.
		Thus $\bar{f}_i=P_1(g)$ for an appropriate $P_1\in k[\bar{f}_0,\ldots,\bar{f}_{i-1}]$.
		On the other hand, since $\bar{f}_i$ is by construction another generator of the field extension $F_{i+1}/F_{i}$, we have similarly $g=P_2(\bar{f}_i)$ for an appropriate $P_2\in k[\bar{f}_0,\ldots,\bar{f}_{i-1}]$. Composing $P_1$ and $P_2$ it is easy to see that $P_1\circ P_2=\textrm{id}$.
		But this is possible only if  $P_1$ is  linear on the $g$ variable, i.e
		\begin{equation}\label{primetop}
			\bar{f}_i=\alpha g +\beta, \textrm{ for some } \alpha, \beta \in k[\bar{f}_0,\ldots,\bar{f}_{i-1}].
		\end{equation}
		Recall that all the pole numbers  in $H(Q_{i+1})$ that arise as the polar part of the functions $\bar{f}_0,\ldots,\bar{f}_{i-1}$ are coming from the push forward of the $H(Q_{i})$ multiplied by $p^{h_{i-1}-h_{i}}$ via the map $\pi^*$ of lemma \ref{up-down}.
		Notice  now that there are only two possible cases:
		\begin{enumerate}
			\item If $\alpha \notin k^*$ or $-v_{Q_{i+1}}(\beta) >-v_{Q_{i+1}}(g)=m$ then the two summands on the right hand side of  eq. (\ref{primetop}), must have equal valuations.
			      If not we contradict our hypothesis $\frac{\bar{m}_i}{|\ker\rho_{c_{i+1}}|}\leq m$.
			      With this in mind,  we get that $m$ is a multiple of $p^{h_{i-1}-h_{i}}$, which  again contradicts our hypothesis.
			\item If $\alpha \in k^*$ and $-v_{Q_{i+1}}(\beta) <m$ then $\frac{\bar{m}_i}{|\ker\rho_{c_{i+1}}|} =m$, compare also to remark \ref{nonuniq}.
		\end{enumerate}
		With another simple argument we will now show that $\bar{m}_i$ is the \textit{only} extra generator of $H(Q_{i+1})$ compared to $H(Q_{i})$.
		Suppose not, and let  $h\in k[\bar{f}_0,\ldots,\bar{f}_{i-1}, \bar{f}_{i}]$ be a rational function such that $(h)_\infty=nQ_{i+1}$ with
		$\frac{\bar{m}_i}{|\ker\rho_{c_{i+1}}|}<n$ a minimal generator of $H(Q_{i+1})$. Again we will view $h$ as a polynomial in $\bar{f}_{i}$. Note that the degree of $h$
		with respect to this variable is less than  $p^{h_{i-1}-h_{i}}$, since $\bar{f}_{i}$ generates the extension. Write
		\[
			h=\sum_{\nu=0}^{p^{h_{i-1}-h_{i}}-1}\alpha_\nu \bar{f}_{i}^\nu, \textrm{ with } \alpha_\nu \in k[\bar{f}_0,\ldots,\bar{f}_{i-1}].
		\]
		All the summands have different valuations. Indeed if this was not the case, then there are indices $s \lneq j$ such that $v_{Q_{i+1}}(\alpha_s \bar{f}_{i}^s)=v_{Q_{i+1}}(\alpha_j \bar{f}_{i}^j)$, or
		\[
			p^{h_{i-1}-h_{i}}\cdot \delta = \frac{\bar{m}_i}{|\ker\rho_{c_{i+1}}|} (j-s), \textrm{ for some  positive integer } \delta.
		\]
		This is impossible since $(j-s)<p^{h_{i-1}-h_{i}}$ and $\frac{\bar{m}_i}{|\ker\rho_{c_{i+1}}|}$ is prime to $p$.
		In this way we manage to write  $-v_{Q_{i+1}}(h)$  as an $\mathbb{N}$-linear combination of smaller minimal generators of the Weierstrass semigroup. This  implies
		that $-v_{Q_{i+1}}(h)$ itself cannot be a minimal generator.
		\end{proof}
		According to proposition \ref{jumpsrho-gen}, since  $\{c_1,\ldots,c_{n}\}$ are the jumps of the representation
		filtration, the elements  $\{\bar{m}_1,\ldots,\bar{m}_n=m_r\}$ are generators of the Weierstrass semigroup $H(P)$.
But it is not true that every generator of $H(P)$ occurs this way, as we have already seen in the examples of this section and as the following lemma indicates:
		\begin{lemma}\label{weirgenbutnotrepjump}
			Let  $M$ be  a {minimal} generator of the Weierstrass semigroup  at $P$
			such that $M \neq \bar{m}_{\nu}$ for all    $1 \leq \nu \leq n$.
			Then any function $f_{M} \in F$ with $(f_M)_\infty=M$ is  $G_1(P)$-invariant.
The number of representation jumps is \textit{either} equal to the number of minimal generators of the Weierstrass semigroup \textit{or} it is equal to the number of minimal generators of the Weierstrass semigroup minus one {and $|G_1(P)|=M$}.
		\end{lemma}
		\begin{proof}
			If there is  such a generator $M_i$ of $H(Q_i)$,
			then this generator is a multiple of a generator of  $H(Q_{i-1})$ by lemma \ref{semigroupup}.
			This means that any  function $f_{M_i}\in F_i$ which has pole number $M_i$ at $Q_i$, is an element invariant under the Galois group of the
			extension $F_{i}/F_{i-1}$.
			Using this argument
			inductively we arrive to the conclusion that  the function $f_{M}$ is $G_1(P)$ invariant and thus, by corollary \ref{div-inv},  $|G_1(P)|$ divides  $M$ and thus $M=|G_1(P)|$.
Finally, if such an $f_{M}$ exists it is unique since $F^{G_1(P)}$ is rational
by our hypothesis. This completes the proof.
		\end{proof}
		We sum up all the information concerning the Weierstrass semigroups of the field tower arising from the representation filtration in the next
		\begin{proposition}\label{weiersemi}
			The Weierstrass semigroups of the fields $F_i$ at $Q_i=P \cap F_i$ for every  $1\leq i \leq n$
			and $\ker \rho_{c_1}=G_1(P)$ are given by
			\[
				H(Q_{i+1})= \langle {\bar{m}_j p^{-h_i}}, |G_1(P)| p^{-h_i}\rangle=
				\left\langle
				\frac{m_{c_j+1}}{|\ker \rho_{c_{i+1}}|},\left| \frac{G_1(P)}{\ker \rho_{c_{i+1}}}\right|  \right\rangle_{\mathbb{Z}_+},
			\]
			where $j$ runs through the indices $1\leq j \leq i$.
			For the Weierstrass semigroup at $P$ we get
			\[
				H(P)=\langle \bar{m}_j, |G_1(P)|\rangle_{\mathbb{Z}_+}, \textrm{ where } 1\leq j \leq n, \textrm{ while } H(Q_1)=\mathbb{Z}_+.
			\]
		\end{proposition}
		\begin{proposition} \label{action}
			Assume that $\sigma \in  \mathrm{ker}\rho_{c_{i}} -  \mathrm{ker}\rho_{c_{i+1}}$.
			Then
			\begin{eqnarray}
				\sigma(f_\nu)& =& f_\nu \mbox{ for all } \nu \leq c_i  \nonumber\\
				\sigma(f_{c_i+1})=\sigma(\bar{f}_{i})& =& \bar{f}_{i}+c(\sigma), \mbox{ where } c(\sigma) \in k ^*.  \label{f-el-ab}
			\end{eqnarray}
		\end{proposition}
		\begin{proof}
			In general $\sigma(\bar{f}_i) = \alpha \cdot \bar{f}_{i}+c(\sigma)$, where $c(\sigma)\in k[f_1,\ldots f_{c_i}],$
			and $\alpha \in k^*$. Since $\sigma$ has order   a power of $p$ we see that  $\alpha=1$.
			But if $c(\sigma)$ is not constant then it has a root $Q\neq Q_i$, 
since the field $k$ is assumed to be algebraically closed.			We will prove that $Q$ is then a ramified point and this will  lead to a contradiction
			since only one place can ramify, and this is $Q_i$.

			Consider the ring $A:=\mathcal{O}(X-Q_i)$, where $\mathcal{O}$ denotes the structure
			sheaf of a nonsingular projective model of our curve $X$ that corresponds to the function field ${F_i}$.  
			The ring $A$ is
			by definition
			\begin{equation*} \label{algebra}
				A= \bigcup_{\nu=0}^\infty L(\nu Q_i)=k[f_1,\ldots,f_{c_i}],
			\end{equation*}
			where the elements $f_1,\ldots,f_{c_i}$ are subject to several relations coming from
			the function field of the curve.
			Observe that when $\nu$ becomes greater than or equal to
			$\bar{m}_{i-1}p^{-h_{i-1}}$
			(i.e. is  greater than all the generators of the Weierstrass semigroup at $Q_i$)
			the algebra generated by $f_1,\ldots,f_{c_i}$ as elements of the vector space $L(\nu Q_i)$ is the ring $A$.
			Keep in mind that the vector space $L(\nu Q_i)$ is inside the function field of the curve
			so there is a well defined notion of multiplication on elements of $L(\nu Q_i)$.
			Every place $Q\neq  Q_i$ of the function field
			$F_i$
			corresponds to
			a unique maximal ideal of the ring  $A$.

			Notice also that the automorphism group acts on $A$. We will prove that the ideal
			$Q$ is left invariant under the action of $\sigma$.
			Let $Q$ be a root of $c(\sigma)$
			and denote by $Q$ the corresponding ideal of $A$. It is finitely generated,
			so $Q=\langle g_j \rangle$
			where $g_j$ are polynomial expressions in $f_i$, where $1\leq i \leq c_i$. We will prove that
			\[
				\sigma(g_j) \in Q \mbox{ for all } j.
			\]
			Indeed, write
			\[
				g_j = \sum_{\nu_1,\ldots,\nu_{c_i}} a_{\nu_1,\ldots,\nu_{c_i}} f_1^{\nu_1}\cdots f_{c_i}^{\nu_{c_i}}.
			\]
			Then
			\[
				\sigma(g_j)=\sum_{\nu_1,\ldots,\nu_{c_i}} a_{\nu_1,\ldots,\nu_{c_i}} f_1^{\nu_1}\cdots \left (f_{c_i}+c(\sigma ) \right )^{\nu_i}=
			\]
			\[
				\sum_{\nu_1,\ldots,\nu_{c_i}} a_{\nu_1,\ldots,\nu_{c_i}} f_1^{\nu_1}\cdots f_{c_i}^{\nu_{c_i}} + 
				\sum_{\nu_1,\ldots,\nu_{c_i}} a_{\nu_1,\ldots,\nu_{c_i}}
				 \sum_{\mu=1}^{\nu_i} f_1^{\nu_1}\cdots f_{c_{i}-1}^{\nu_{c_{i}-1}}
				\binom{\nu_{c_i}}{\mu}  c(\sigma)^{\mu} f_{c_i}^{\nu_{c_i}-\mu}.
			\]
			But $Q$ is a root of $c(\sigma)$ and this is equivalent to $c(\sigma) \in Q$ so the second summand  of the last equation is an element in $Q$.

			We would like also to point out how  we can construct  the curve $X-Q_i$.
			If $\nu$ is big enough then
			the projective map $\Phi$ corresponding to the linear series
			$|\nu Q_i|$ is an
			embedding, see for example \cite[Theorem 4.3.15]{Goldschmidt}. The image $\Phi(X)$ is then a nonsingular curve; removing the point
			$\Phi(Q_i)$ we obtain the affine non-singular curve with coordinate ring $A$. Notice that, by construction, $X$ is the projective
			closure of that curve with $Q_i$ being the point at infinity, while the function fields for both curves are just
			${F_i}$.
		\end{proof}
		In what follows we will use the following
		\begin{lemma}\label{HKT}
			Let   $f\in F$ such that $p\nmid v_P(f)$. If $\sigma \in G_i \setminus G_{i+1}$, then
			$\sigma(f)=f+f'$ with $f'\neq 0$ and $i=-v_P(f)+v_P(f')$.
		\end{lemma}
		\begin{proof}
			This is  \cite[Lemma 11.83]{hirschfeld2008algebraic}.
		\end{proof}
		\begin{theorem}\label{jumps}
			Let $P$ {be} the totally ramified place of the HKG-cover. Recall that  $Q_i=P\cap F_i$, with  $1\leq i \leq n+1$. Let $\mu$ be  the number of ramification jumps of eq. (\ref{ramfiltration}) and $n$ the number of representation jumps, see eq. (\ref{rep-jumps}).
			\begin{itemize}
				\item[(i)]\label{A} The groups  $\ker \rho_{c_{i}}/\ker \rho_{c_{i+1}}$, for each $1\leq i \leq n$,  have exactly one  lower ramification jump
				which  is equal to $-v_{Q_{i+1}}(\bar{f}_i)$.
				\item[(ii)] The jumps {mentioned in (i)} are equal to the ramification jumps of the groups $G_{b_i}/G_{b_{i+1}}$, for $1\leq i \leq \mu$,  thus  $\mu=n$ and they exhaust  all the ramification jumps of $G_1(P)$. 

			\end{itemize}
		\end{theorem}
		\begin{proof}
			We first prove (i).
			Lemma \ref{semigroupup} implies
			$\textrm{gcd}(v_{Q_{i+1}}(\bar{f}_i),p)=1$.
			Using proposition \ref{action} and lemma  \ref{HKT} we obtain that  the jump for
			$\ker \rho_{c_{i}}/\ker \rho_{c_{i+1}}=\mathrm{Gal}(F_{i+1}/F_i)$
			is indeed $-v_{Q_{i+1}}(\bar{f}_i)$ since $\sigma(\bar{f}_i)=\bar{f}_i+c(\sigma)$, where $c(\sigma)$ is constant
			and has valuation $0$.
			This jump is also unique by lemma \ref{HKT}.
Moreover the extension $F_{i+1}/F_i$ is elementary abelian since
 $c:G_1(P)\rightarrow k$ gives rise to  an isomorphism from $\mathrm{Gal}(F_{i+1}/F_i)$ to a $p$-subgroup of the additive group of $k$. 
Compare also to \cite[Prop III.7.10]{StiBo}.


			In order to prove (ii) we are going to apply (i) step by step. In the first step we consider the
			group $\ker \rho_{c_{n}}$ which is elementary abelian
			with a unique jump at $m_r$. Since this group is a subgroup of $G_1(P)$ and this is the maximum jump we can have  (see  proposition \ref{gap-comp}), we  obtain $m_r=b_\mu$.

			For the next step we consider the lower  ramification jumps of the filtration of the group $\ker \rho_{c_{n-1}}$.
			From the previous step we see that the quotient
			$\ker \rho_{c_{n-1}}/\ker \rho_{c_{n}}$ has a unique lower jump at $-v_{Q_{n}}(\bar{f}_{n-1})$.  It is well known that the first
			jumps in the lower and upper numbering coincide, since the Herbrand function $\phi$, as it is defined in \cite[IV.3, p.73]{SeL}, is the identity for values smaller than the first lower jump. Thus $-v_{Q_{n}}(\bar{f}_{n-1})$ is also the first jump in the upper numbering for $\ker \rho_{c_{n-1}}/\ker \rho_{c_{n}}$. Using the well known property of the upper ramification filtration: that for all normal subgroups $H$ of $G$ and $u$ an upper jump we have $(G/H)^u=G^u H/H$  \cite[IV. prop. 14]{SeL}, we  derive that $-v_{Q_{n}}(\bar{f}_{n-1})$ equals also the first upper, thus the first lower jump of $\ker \rho_{c_{n-1}}$. Note that
			since $\ker \rho_{c_{n}}$ is a normal subgroup of $\ker \rho_{c_{n-1}}$, the latter group inherits the lower ramification jump of the
			first step. That is, $m_r$ is also a lower jump for $\ker \rho_{c_{n-1}}$, the greatest one, since by eq.  (\ref{bas-val1}) we have
\[
-v_{Q_{n}}(\bar{f}_{n-1})=
\frac{m_{c_{n-1}+1}}{|\ker \rho_{ c_{n}}|}=\frac{\bar{m}_{n-1}}{|\ker \rho_{ c_{n}}|} < 
\frac{\bar{m}_n}{|\ker{\rho_{c_{n+1}}}|}=-v_{Q_{n+1}}(\bar{f}_{n}).
\]
Notice that  $\bar{m}_n= m_{c_n +1}=m_{r}$ and $|\ker{\rho_{c_{n+1}}}|=1$.

			We continue like this,
			using the fact that every ramification jump of a subgroup of $G_1(P)$ is a
			ramification jump of $G_1(P)$ as well \cite[Proposition IV.2 p. 62]{SeL} and  get that all the
			positive integers $-v_{Q_{i+1}}(\bar{f}_i)$  are indeed jumps of $G_1(P)$.

			Are there more jumps of the ramification filtration?
			By construction $\ker \rho_{c_1}=G_1(P)$ and $\ker \rho_{c_1}$ has at least $n$ lower ramification jumps, since $n$ is the number of representation jumps and by (i) every representation jump gives rise to a lower ramification jump. 
			If the number of the ramification jumps was strictly greater than $n$, then some   of the Galois groups $\ker \rho_{c_i}/\ker \rho_{c_{i+1}}$ 
should have more than one lower ramification jump 
			which is  impossible from the computations done above.
		\end{proof}

		\begin{corollary}\label{sameorder}
The following  groups are equal:
			\[G_{b_{i}}=\ker \rho_{c_i} \mbox{ for all  } 1\leq i \leq \mu=n.\]
		\end{corollary}
		\begin{proof}
			We will prove  first that $\ker \rho_{r-1} \subset G_{b_\mu}$.
			But $b_\mu=m_r$, thus
			for an  element $\sigma\in \ker \rho_{r-1}$ we have
			$\sigma(f_r)=f_r+ c(\sigma)$, with $c(\sigma)\in k^*$ so
			\[v_P(\sigma(t) -t)=m_r+1 = b_\mu+1 \Rightarrow \sigma \in G_{b_\mu}.\]
			Now we will prove that $\ker \rho_{r-1} \supset G_{b_\mu}$.

			Notice that every element in $G_{b_\mu}$ satisfies $v_P(\sigma(t)-t)=b_\mu+1=m_r+1$.
			Let $c_{i^*}$ be maximal
			such that $G_{b_\mu} \subset \ker \rho_{c_{i^*}}$. Then  by construction, there is an element $\sigma' \in G_{b_\mu}$
			that does not belong to $\ker\rho_{c_{i^*+1}}$, that is (using proposition \ref{action})
			\[
				\sigma'(f_{j})=f_{j} \mbox{ for all } j \leq c_{i^*}  \mbox{ and } \sigma'(f_{c_{i^*}+1})=f_{c_{i^*}+1}+\sigma'(c), \textrm{ for some } \sigma'(c)\in k^*.
			\]
			For a Galois group $G$ of a local field extension $L/K$ consider the function $i_{G}$ defined by $i_{G}(\sigma)=v_{L}(\sigma(t)-t)$, see \cite[chap IV, p.62]{SeL}.  We consider this function for the Galois extension  
			$\frac{\ker \rho_{c_{i^*}}
			 	}{
			 	\ker \rho_{c_{i^*+1}}}$:			
			\[
i_{\frac{\ker \rho_{c_{i^*}}
				}{
				\ker \rho_{c_{i^*+1}}}}
			\left(
			\sigma' \ker \rho_{c_{i^*+1}}
			\right)=-v_{Q_{i^*+1}}(f_{c_{i^*}+1})+1,
			\]
			using lemma \ref{HKT}.
			 On the other hand this value should be equal to  $b_\mu$. Notice, that
			since $G_{b_\mu}$ is elementary abelian with a unique jump, the lower and the  upper ramification filtrations coincide.
			So $m_r=b_\mu=-v_{Q_{i^*+1}}(f_{c_{i^*}+1})$. Thus  
			{$i^*=n$} and $c_{i^*}=c_n=r-1$. 
			This proves that $\ker \rho_{r-1} =G_{b_\mu}$, i.e. the last groups in both filtrations coincide.

			We now consider the HKG extension of the rational function field given by:
			\[F^{G_{b_\mu}}=k (X/{\ker \rho_{c_{n}}})=F^{G_1(P)}(\bar{f}_1,\ldots,\bar{f}_{n-1}).\]
			This extension, has ramification filtration
			\[
				\frac{G_1(P)}{G_{b_\mu}} \geq \cdots \geq \frac{G_{i}}{G_{b_\mu}} \geq  \cdots \geq \frac{G_{b_\mu-1}}{G_{b_\mu}} > \{1 \}.
			\]
			Indeed, we  know by \cite[Corollary on page 64]{SeL} that the ramification filtration
			of the quotient group $G/H$ when $H=G_j$ is a subgroup of the ramification filtration
			is given by $(G/H)_i=G_i/H$ for $i\leq j$ and $(G/H)_i=\{1\}$ for $i\geq j$.
			The representation filtration of $\frac{G_1(P)}{G_{b_\mu}}$ is formed by the quotients of the representation filtration of $\ker \rho_{c_1}$ by  $\ker \rho_{r-1}$. Using the previous argument we see
			that the last groups in both filtrations are equal and we proceed inductively  using theorem \ref{jumps}.
		\end{proof}

		We will now focus on the case where the first jump equals one:
		\begin{corollary}  \label{rational-jump}
			The condition $G_1(P)> G_2(P)$ is  equivalent to   ${F_2}$ being rational.
		\end{corollary}
		\begin{proof}
			Let $[G_1(P):\ker\rho_{c_2}]=:q$. The group $G_1(P)/\ker \rho_{c_2}$
			is elementary abelian of order $q$ with a unique jump, say at $\upsilon$. The Riemann--Hurwitz
			theorem implies:
			\[
				2g_{F_2} -2=-2q+(\upsilon+1)(q-1)
			\]
			and $\upsilon=1$ if and only if  $g_{F_2}=0$.
		\end{proof}

		\begin{corollary}\label{case1jumps}
			Suppose that $G_1(P)> G_2(P)$.  Let $i_0$ be the index such that $-v_P(f_{i_0})=m_{i_0}=|G_1(P)|$ and
			$k(f_{i_0})=F^{G_1(P)}$ as  given in lemma \ref{fields}.
			Concerning the structure of the Weierstrass semigroups $H(Q_{i+1})$ given in proposition \ref{weiersemi} we have
			\[
				H(Q_{i+1})=\left\langle \bar{m}_jp^{-h_j} : 1\leq j \leq i \right\rangle_{\mathbb{Z}_+},
			\]
			while
			\[
				H(P)=\langle \bar{m}_{j}: 1\leq j \leq n\rangle_{\mathbb{Z}_+}.
			\]
			More precisely, $|G_2(P)|=m_1$, i.e. the order of the second lower ramification group
			equals the first pole number and
			\[ m_r =m_{r-1}+1.\]
		\end{corollary}
		\begin{proof}
			The element  $f_{i_0}$   is not needed for the
			generation of $F_j=F^{\ker \rho_{c_{j}}}$ for every $j>1$, that is
			$\left\langle \bar{m}_{j}p^{-h_i}  \right\rangle_{\mathbb{Z}_+} \ni \left|\frac{G_1(P)}{\ker \rho_{c_{i+1}}}\right|$.
			Indeed,
			from corollary \ref{rational-jump}
			we have  $F^{G_1(P)}(\bar{f}_1)=F_2$ is rational. 
			The element $f_{i_0}$
			is  a rational function {of}  $\bar{f}_1$.
			Moreover in this case,  we can normalize the Artin--Schreier generator $\bar{f}_1$
			for the  elementary abelian extension
			with a unique ramification jump, and  apply \cite[Proposition III.7.10]{StiBo} such that
			\[
				f_{i_0}=\bar{f}_1^q -\bar{f}_1,
			\]
			where $q$ equals  $[G_1(P):\ker\rho_{c_2}]$.

			Corollary \ref{g0} implies that   $|G_1(P)|$
			can result as a pole number as a multiple  of $|\ker \rho_{c_2}|$, which is a pole number since $F^{\ker \rho_{c_2}}={
				{F}_2={F}_1(\bar{f_1})}$  is rational. Moreover,
			from corollary \ref{sameorder} we have that $|G_2(P)|=|\ker \rho_{c_2}|$, while
			$|\ker \rho_{c_2}|=\bar{m}_1$ and thus
			\[
				\left| \frac{G_1(P)}{\ker \rho_{c_{i+1}}}\right| \in  \left\langle \frac{\bar{m}_1}{|\ker \rho_{c_{i+1}}|}\right\rangle_{\mathbb{Z}_+}, \textrm{ for every }
				1\leq i \leq n.
			\]
			Notice that in this case $\bar{m}_{1}=m_1$,  and that the first non zero pole number is always a {minimal} generator.

			Finally the last assertion about $m_r$ comes directly from proposition \ref{gap-comp}.
		\end{proof}
		
		At this point, we would like to discuss the case where $|G_1(P)|$ is  a generator of the semigroup.
		It turns out that this happens if and only if $1$ is  \textit{not}  a ramification jump, i.e. $G_1(P)=G_2(P)$.
		We have seen that the {minimal} generators of the semigroup $H(P)$ are of two types:
		\begin{enumerate}
			\item  they are induced by jumps of the representation filtration
			\item $|G_1(P)|$.
		\end{enumerate}
		{We will need the following}
			\begin{lemma} \label{Smith}
				Assume that $S$ is a numerical semigroup and $E$ is the semigroup such that $E=p^\ell S+N{\mathbb{Z}_+}$, for some $\ell\in  \mathbb{N}$, where
				$(N,p)=1$. Suppose further that
				the semigroups $S$, $E$ have the same cardinality of minimal generators. Then $N$ is a generator of the
				semigroup $S$.
			\end{lemma}
			\begin{proof}
				This is  \cite[proposition A.0.15]{fractionsemi} in the PhD thesis of  H. Smith. Notice that there the result is proved only for
				$p^\ell=p$, but the same proof can be used for the more general case of higher values of $\ell$.
			\end{proof}
		\begin{proposition} \label{gen-chara}
			The number $|G_1(P)|$ is a {minimal} generator of the Weierstrass semigroup at $P$ if and only if $G_1(P)=G_2(P)$.
		\end{proposition}
		\begin{proof}
			If $G_1(P) > G_2(P)$, $F^{G_2(P)}$ is rational, $|G_2(P)|$ equals the first pole number from corollary \ref{case1jumps} and since
			$|G_2(P)|$ divides $|G_1(P)|$, $|G_1(P)|$  cannot be  a minimal
			   generator.

			For the other direction, assume that  $|G_1(P)|$ is not a  minimal generator, then we will prove that  $G_1(P)> G_2(P)$.
			By our hypothesis,  there is a semigroup $H(Q_{i})$ where $|G_1(P)|/| \ker \rho_{c_{i}}|$
			is not a generator  for some $c_i<r$. Let $\nu^*$ be the first index such that $|G_1(P)|/| \ker \rho_{c_{i}}|$ is
			a generator for $i\leq \nu^*$ and $|G_1(P)|/| \ker \rho_{c_{\nu^*+1}}|$ is not  a generator for
			$H(Q_{\nu^*+1})$.
			We have the following generating sets for the semigroups:
			\[
				H(Q_{\nu^*})=\left\langle  \left| \frac{G_1(P)}{\ker \rho_{c_{\nu^*}}}\right|,\frac{\bar{m}_{j}}{|\ker \rho_{c_{\nu^*}}|}: 1\leq j < \nu^* \right\rangle_{\mathbb{Z}_+},
			\]
			\[
				H(Q_{\nu^*+1})=\left\langle  \frac{\bar{m}_j}{|\ker \rho_{c_{\nu^*+1}}|}: 1\leq j \leq \nu^* \right\rangle_{\mathbb{Z}_+},
			\]
			i.e. both semigroups have the same number of generators.
			According to lemma \ref{semigroupup}  the semigroup   $H(Q_{\nu^*+1})$ is generated by elements of the semigroup  $H(Q_{\nu^*})$ multiplied by
			$[\ker \rho_{c_{\nu^*}} : \ker \rho_{c_{\nu^*+1}}]$ and an extra {prime to $p$} generator
			$\frac{\bar{m}_{\nu^*}}{|\ker \rho_{c_{\nu^*+1}}|}$, i.e.:
			\[
				H(Q_{\nu^*+1})=[\ker \rho_{c_{\nu^*}}:\ker \rho_{c_{\nu^*+1}}]\cdot H(Q_{\nu^*}) +\Z_+
				\frac{\bar{m}_{\nu^*}}{|\ker \rho_{c_{\nu^*+1}}|}.
			\]
			We will now complete the proof, by applying lemma \ref{Smith}.
			The prime to $p$ generator $N=\frac{\bar{m}_{\nu^*}}{|\ker \rho_{c_{\nu^*+1}}|}$ should be a generator
			of $H(Q_{\nu^*})$ but it cannot be any of the $\frac{\bar{m}_j}{|\ker \rho_{c_{\nu^*}}|}: 1\leq j < \nu^*$
			since it is the greatest of these, so the only remaining case is $N=\left|\frac{G_1(P)}{\ker \rho_{c_{\nu^*}}}\right|$,
			but since $N$ is prime to $p$ we have $|G_1(P)|= |\ker \rho_{c_{\nu^*}}|$,  $N=1$ and thus $\nu^*=1$ and
			$H(Q_1)=H(Q_2)=\Z_+$,
			{but this} contradicts the non-rationality of the field $F^{G_2(P)}$.
		\end{proof}
		\begin{remark}
			For HKG-covers the field $F^{G_2(P)}$ is always rational, see \cite[Theorem 11.78 (iii)]{hirschfeld2008algebraic}.
		\end{remark}
		\begin{remark}[Upper ramification jumps]
			The reader should notice that by computing the jumps of the lower ramification filtration
			we gain information on the jumps of the \textit{upper} ramification filtration through the Herbrand's formula,
			see \cite[section IV]{SeL}. As an application of this we get that, for $p$-groups, upper and lower ramification jumps are connected by
			the following formula:
			\[
				b_i=\sum_{j=1}^i (u_j-u_{j-1}) p^{h_0 -h_{j-1}} , \textrm{ for every } 1\leq i \leq n,
			\]
			where $u_1,\ldots,u_n$ are the upper jumps of $G_1(P)$ and here $b_0=u_0=0.$
	\end{remark}

\section{Applications}
\label{apply}
\subsection{Big actions}
\label{bigactionsap}

		A case where the order of $G_1(P)$ is not a generator of $H(P)$, due to Proposition \ref{gen-chara}, is when we focus on {\em big actions} as this { notion} is   defined in the
		work of C. Lehr, M. Matignon \cite{CL-MM} and studied further by M. Rocher and M. Matignon
		\cite{Ma-Ro}, \cite{Ro}.

	\begin{definition}
		A curve $X$ together with a subgroup $G$ of  the automorphism group of $X$
		is called a big action if $G$ is a $p$-group and
		\[
			\frac{|G|}{g_X} > \frac{2p }{p-1}.
		\]
	\end{definition}

		All big actions have the following property 
		\begin{proposition}\cite[prop. 8.5]{CL-MM}: \label{28}
			Assume that $(X,G)$ is a big action.
			There is a unique  point $P$ of $X$ such that $G_1(P)=G$, the group $G_2(P)$ is not trivial and strictly contained in $G_1(P)$ and the quotient $X/G_2(P)\cong \mathbb{P}^1$. Moreover, the
			group $G$ is an extension of groups
			\[
				0 \rightarrow G_2(P) \rightarrow G=G_1(P) \stackrel{\pi}\longrightarrow (\Z/p\Z)^v \rightarrow 0.
			\]
		\end{proposition}
		The first jump for their ramification filtration {is equal to} 1, while the other {jumps} are given by theorem \ref{jumps}. Moreover, we are now able able to compute explicitly the Weierstrass semigroup at the ramified point.
			\begin{corollary}\label{bigactions}
				If $(X,G)$ is a big action, then
				$|G_1(P)|$ is not a minimal generator of $H(P)$. Moreover
				\[
					H(P)=\langle \bar{m}_{j}: 1\leq j \leq n\rangle_{\mathbb{Z}_+},\; |G_2(P)|=m_1 \textrm{ and } m_r =m_{r-1}+1,
				\]
				i.e. the structure of $H(P)$ is given by corollary \ref{case1jumps}.
			\end{corollary}

\subsection{Curves with zero $p$-rank}
\label{zeroprankapp}
The $p$-rank of the Jacobian is an important invariant of an algebraic curve, which also controls the automorphism group of the curve, see \cite{Nak}. 
The case of zero $p$-rank curves corresponds to curves $X$ with a huge number of automorphisms
\cite[Theorem 1 (iv)]{Nak}. In this class of curves, the  most
		automorphisms occur exactly when $ X/G_1(P)$  is rational. Otherwise $|G_1(P)|$
		is less than or equal to the genus of the curve see \cite[Theorem 11.78 (i)]{hirschfeld2008algebraic}. This is exactly the HKG $p$-case.

\begin{theorem}\label{zeroprankGK}
			The following conditions are equivalent, for a $p$-group $G\subseteq \A(X)$: 
			\begin{enumerate}
				\item\label{direct} The curve $X$ has zero $p$-rank and $|G|$ is a pole number at the unique  point $P\in X$ that $G$ stabilizes.
				\item\label{inverse} The cover $X\rightarrow X/G$ is a  HKG-cover.
			\end{enumerate}
		\end{theorem}

		\begin{proof}
			$\ref{direct} \Rightarrow \ref{inverse}.$ By \cite[Lemma 11.129]{hirschfeld2008algebraic}
			every element of order $p$ fixes exactly one point. This means that $G=G(P)$, i.e. $G$ can be realized as the stabilizer of a point $P\in X$
			and that
			for the cover $X\rightarrow X/G(P)$, $P$ is the unique totally ramified point.
			By corollary \ref{g0},  $|G|=|G(P)|$ is a pole number at $P$ if and only if $X/G(P)$ is a rational curve.

			$\ref{inverse} \Rightarrow \ref{direct}.$ Use the Deuring--Shafarevich formula \cite[eq. (1.1)]{Nak:85} and the definition of a HKG $p$-cover.
		\end{proof}

%
%
%
%
%

\subsection{Hasse-Arf divisibility conditions}
\label{HasseArfapp}

 The Hasse--Arf theorem for abelian groups gives certain divisibility conditions for the jumps 
 of the ramification filtration. Using theorem \ref{mainres} restricted to the case of  an abelian group $G_1(P)$,
 these divisibility conditions can  be interpreted in terms of the Weierstrass semigroup at $P$:

\begin{corollary}[Hasse--Arf theorem]
Assume that a HKG-cover has abelian Galois $p$-group $G_1(P)$. Let $p^{r_i}=[G_{b_{i}}:G_{b_{i+1}}]$ for all $1\leq i \leq n-1$. Then the generators of the Weierstrass semigroup 
that result from the jumps of the representation filtration satisfy:
 \[
\frac{{\bar{m}_{i+1}}}{|G_{b_{i+2}}|} \equiv \frac{{\bar{m}_{i}}}{|G_{b_{i+1}}|}  \mod \; p^{\sum_{j=1}^{i}r_j}
 \]
 or
 \[
  \left |\frac{G_{b_{i+1}}}{G_{b_{i+2}}}\right|  {\bar{m}_{i+1}} \equiv {\bar{m}_{i}} \mod \left| G_{b_1}\right |.
 \]
\end{corollary}
\begin{proof}
We will use an equivalent form of the Hasse--Arf theorem, see \cite{rouketas}: the condition for the upper jumps $u_i$ to be integers
can be directly translated to congruences for the lower ramification jumps.
Namely, every two subsequent lower ramification jumps $b_{i+1}, b_i$  must satisfy:
\[
b_{i+1}\equiv b_i \mod p^{\sum_{j=1}^{i}r_j}, \textrm{ where } p^{r_i}:=[G_{b_{i}} : G_{b_{i+1}}], \textrm{ for every } 1\leq i \leq n-1.
 \]
 Now replace $b_i$ with $\frac{{\bar{m}_{i}}}{|G_{b_{i+1}}|}$  for every $1\leq i \leq n$ in order to derive the desired result.
\end{proof}

%
\section{Holomorphic polydifferentials}
\label{represent-KG}
In what follows $X$ is always a HKG-cover with Galois group a $p$-group.
			We can construct a basis for the $m$-holomorphic polydifferentials of $X$ as
			follows:

			Let $f_{i_0}$ be the function generating the rational field $F^{G_1(P)}=k(f_{i_0})$. The function
			$f_{i_0}$ can be selected so that it has a simple  unique pole  at infinity which is the restriction of the place $P$ to $k(f_{i_0})$.
			Let $p^{h_0}=|G_1(P)|$.
			We observe first that
			\begin{equation} \label{deg-cont}
				\mathrm{div} (df_{i_0}^{\otimes m})= \left( -2 m p^{h_0} + m\sum_{i= 1}^n (b_i -b_{i-1})(p^{h_{i-1}}-1) \right) P,
			\end{equation}
			where
			\[
				b_0=-1,\; p^{h_0}=|G_1(P)|, \; p^{h_i}=|\ker \rho_{c_{i+1}}|=| G_{b_{i+1}}|, \textrm{ for } i\geq 1.
			\]
			The right hand side of eq. (\ref{deg-cont})  equals  $m (2g_X-2)P$ by the Riemann--Hurwitz formula.
			\begin{proposition} \label{deg-char}
				For every pole number $\mu$, we select a function $f_\mu$ such that $(f_\mu)_\infty=\mu P$.
				The set $\{f_\mu df_{i_0}^{\otimes m}: \deg\mathrm{div}(f_i)\leq  m(2g_X-2)\}$ is a basis for the
				space of $m$-holomorphic (poly)differentials of $X$, for every positive integer $m\geq1$.
			\end{proposition}
			\begin{proof}
				All $m$-holomorphic differentials are of the form $ g df_{i_0}^{\otimes m}$. Therefore, the
				condition for being holomorphic is translated into the condition
				$g \in L( m (2g_X-2)P)$.
				This means that the linear independent elements $f_i df_{i_0}^{\otimes m}$ with $\deg \mathrm{div} f_i=m_i \leq m(2g_X-2)$
				are holomorphic. In order to see that all the holomorphic differentials are of this form, we will count them:

				\textit{Case $m=1$.} Notice that $\ell((2g_X-2)P)=g_X$.
				On the  other hand $\ell((2g_X-1)P)=g_X$ from the Weierstrass gap theorem \cite[I.6.7]{StiBo}. This means
				that in the interval $[0,2g_X-2]$ there are exactly $g_X$ pole numbers, equivalently $2g_X-1$ is a gap.

				\textit{Case $m>1$.} Similarly, observe {using the Riemann--Roch theorem}, that the space of $m$-holomorphic differentials has dimension
				\[
					\dim L(mW)=m(2g_X-2) + 1-g_X=(2m-1)g_X-2m+1.
				\]
				On the other hand the number of $f_i$ such that $\deg \mathrm{div}(f_i) \leq m (2g_X-2)$ can be computed
				as follows:

				In the interval $[0,2g_X-1]$ there are $g_X$ such elements
				and every number greater than $2g_X$ is a pole number using
				again the Riemann--Roch theorem. So
				in the interval $(2g_X-1, m(2g_X-2)]$ there are
				$m(2g_X-2)-(2g_X-1)=2mg_X -2m -2g_X+1$ elements. In total there are
				$ 2mg_X -2m-2g_X+1+g_X=(2m-1)g_X-2m+1$ and this coincides with the dimension of the space of $m$-holomorphic
				differentials.
			\end{proof}
			\begin{corollary}\label{symmetricWeier}
				The Weierstrass semigroup at $P$ is symmetric, i.e. $2g_X-1$ is a gap.
			\end{corollary}
			We have proved in proposition  \ref{weiersemi} that  the elements $m_{c_i+1}$ for    $1\leq i \leq n$ together with the element
			$p^{h_0}$ generate the Weierstrass semigroup. A numerical semigroup $\Sigma$  that is not of the form
			$a \Z_+$ has a minimal element $\kappa(\Sigma)$ called {\em the conductor} such that  all integers
			$n\geq \kappa(\Sigma)$ are in the semigroup.

			Since the semigroup is symmetric we see that $\kappa(H(P))=2g_X$. Recall that $2g_X-1$ is a gap in this case
			and that the  Riemann--Roch theorem implies that all integers $\geq 2g_X$ are in $H(P)$.

We will now focus on the representation theory of HKG-covers.

			\begin{proposition}\label{KGmodule}
				Let $p^{h_0}=|G|=|G_1(P)|$.
				The module $\Omega_X^{\otimes m}$ is the direct sum of at most
				\[
					N:=\lf \frac{m(2g-2)}{p^{h_0}} \rf=-2m +\lf m \frac{ \sum_{i= 1}^n (b_i -b_{i-1})(p^{h_{i-1}}-1)}{p^{h_0}} \rf
				\]
				direct indecomposable summands.
			\end{proposition}
			\begin{proof}
				We have  a representation of the group $G_1(P)$ in terms of lower diagonal matrices in
				$\Omega_X^{\otimes m}\cong L(m(2g_X-2)P)$. For an  element $f$ in  $L(m(2g_X-2)P)$
				we have the function $v_P:L(m(2g_X-2)P) \rightarrow \mathbb{N}$ sending
				$f$ to $-v_P(f)$ and $v_P(\sigma(f)-f)>v_P(f)$.

				Assume that the space $L(m(2g_X-2)P)$ admits a decomposition
				\[
					L(m(2g_X-2)P)= \bigoplus W_i
				\]
				as a direct sum of $G$-modules $W_i$.
				We will prove that we can find a basis of elements $e_1,\ldots e_{\dim W_i}$ of
				$W_i$ that have different valuations. {Indeed, s}tart from any basis of $W_i$. If there are
				two basis elements $a,b$ of $W_i$ such that $v_P(a)=v_P(b)$, then these are, locally at $P$, of the form
				\[
					a=a_1 \frac{1}{t^v}+ \mbox{higher order terms}, b= b_1 \frac{1}{t^v} + \mbox{ higher order terms}.
				\]
				Therefore there is an element $\lambda$ such that $a-\lambda b\neq 0$ has different
				valuation than $a,b$, ($\lambda=a_1/b_1$). We replace the element $b$ by the
				element $a-\lambda b$.  Proceeding this way we construct the desired
				basis  elements with different valuations. Now,
				\[\sigma(e_i)=e_i+b_i(\sigma), \mbox{ with } b_i(\sigma)=0 \textrm{ or } |v_P(b_i(\sigma))| < |v_P(e_i)|\]
				and this proves that every direct summand $W_i$ has an upper triangular representation
				matrix, so it contains at least one invariant element.

				Therefore,
				the number of indecomposable summands is smaller than  the number of $G_1(P)$-invariant elements.
				The space of invariant elements has  a basis of elements of the form $f_{i_0}^j$ such that  $-v_P(f_{i_0}^j) \leq m(2g_X-2)$, and the
				result follows.
			\end{proof}

			\begin{corollary}\label{indpoly}
				If $|G_1(P)|> m(2g_X-2)$ then the module $H^0(X,\Omega^{\otimes m})$ is indecomposable.  In particular  the space of holomorphic differentials $H^0(X,\Omega)$ is indecomposable for a curve $X$ that admits a big action.
			\end{corollary}
			\begin{proof}
				If $|G_1(P)|> m(2g_X-2)$ then the only $G_1(P)$ invariant elements belonging to $L(2m(g_X-1))$ are the constants.
				Thus this space includes a unique copy of the one dimensional irreducible representation, so is indecomposable. The assertion for curves
				admitting big action comes directly now from their definition.
			\end{proof}
			
 \def\cprime{$'$}
\providecommand{\bysame}{\leavevmode\hbox to3em{\hrulefill}\thinspace}
\providecommand{\MR}{\relax\ifhmode\unskip\space\fi MR }
\providecommand{\MRhref}[2]{%
  \href{http://www.ams.org/mathscinet-getitem?mr=#1}{#2}
}
\providecommand{\href}[2]{#2}

\end{document}